\documentclass[11pt,twoside,reqno,centertags,draft]{amsart}
\usepackage{amsfonts}
\usepackage{color,enumitem,graphicx}
\usepackage[colorlinks=true,urlcolor=blue,
citecolor=red,linkcolor=blue,linktocpage,pdfpagelabels,
bookmarksnumbered,bookmarksopen]{hyperref}
  \usepackage{amsmath,amsthm,amsfonts,amssymb}

\begin{document}
\title{ Liouville theorem for the fractional Lane-Emden Equation in unbounded domain  }
\date{}
\maketitle

\vspace{ -1\baselineskip}

{\small
\begin{center}

\medskip

  {\sc   Huyuan Chen\,$^*$}
  \medskip

   Department of Mathematics, Jiangxi Normal University,\\
Nanchang, Jiangxi 330022, PR China\\
   Email: chenhuyuan@yeah.net\\[10pt]

\end{center}
}

\smallskip
{\small

\begin{quote}
{\bf Abstract.}
Our purpose of this paper is to study the nonexistence of nonnegative very weak solutions of
\begin{equation}\label{eq 0.1}
 \displaystyle  (-\Delta)^\alpha   u = u^p+\nu\quad
   {\rm in}\quad  \Omega,\qquad\   u=g\quad  {\rm in}\quad \mathbb{ R}^N\setminus \Omega,
 \end{equation}
where $\alpha\in(0,1)$, $p>0$,
$\Omega$ is a unbounded $C^2$ domain in $\mathbb{ R}^N$ with $N>2\alpha$, $g\in L^1(\mathbb{ R}^N\setminus \Omega,\frac{dx}{1+|x|^{N+2\alpha}})$ nonnegative and $\nu$ is a nonnegative Radon measure.
We obtain that\smallskip

 $(i)$ if $\Omega\supseteq \left(\mathbb{R}^N\setminus \overline{B_{r_0}(0)}\right)$ for some $r_0>0$ and $p<\frac{N}{N-2\alpha}$, then
problem (\ref{eq 0.1})  has no weak solutions.

$(ii)$ if $\Omega\supseteq \left\{x\in \mathbb{ R}^N:\, x\cdot a>r_0\right\}$ for some $r_0\ge 0$, $a\in \mathbb{ R}^N$ and $p<\frac{N+\alpha}{N-\alpha}$,
then problem (\ref{eq 0.1})  has no weak solutions. Here $\frac{N+\alpha}{N-\alpha}$ is sharp for the nonexistence in the half space. \smallskip

The above Liouville theorem  could be applied to obtain nonexistence of classical solution of
the  fractional Lane-Emden equations
$$
   (-\Delta)^\alpha   u = u^p\quad
  {\rm in}\quad  \Omega, \qquad  u\ge 0\quad {\rm in}\quad \mathbb{ R}^N\setminus \Omega,
 $$
where $\Omega=\mathbb{R}^N\setminus B_{r_0}(0)$ with $r_0>0$ or $\Omega=\mathbb{R}^{N-1}\times(0,+\infty)$.
\end{quote}

\medskip
\begin{quote}
{\bf R\'esum\'e.}
Le but de cet article est d'\'etudier \`a la non-existence de solution nonnegative tr\`es faible de
\begin{equation}\label{eq 0.01}
 \displaystyle  (-\Delta)^\alpha   u = u^p+\nu\quad
  {\rm dans}\quad  \Omega,\qquad
 u=g\quad  {\rm dans}\quad \mathbb{ R}^N\setminus \Omega,
\end{equation}
o\`u $\alpha\in(0,1)$, $p>0$,
$\Omega$ est un domaine de $C^2$, nonborn\'e de $\mathbb{ R}^N$ avec $N>2\alpha$, $g\in L^1(\mathbb{ R}^N\setminus \Omega,\frac{dx}{1+|x|^{N+2\alpha}})$ nonnegative et $\nu$ est une mesure de Radon nonnegative.
On obtient alors \smallskip

 $(i)$ si $\Omega\supseteq \left(\mathbb{R}^N\setminus \overline{B_{r_0}(0)}\right)$ pour certain $r_0>0$ et $p<\frac{N}{N-2\alpha}$, alors le
probl\`eme (\ref{eq 0.01})  n'a pas de solution faible.

$(ii)$ si $\Omega\supseteq \left\{x\in \mathbb{ R}^N:\, x\cdot a>r_0\right\}$ pour certain $r_0\ge 0$, $a\in \mathbb{ R}^N$ et $p<\frac{N+\alpha}{N-\alpha}$,
alors le probl\`eme (\ref{eq 0.1})  n'a pas de solution faible. Ici $\frac{N+\alpha}{N-\alpha}$ est optimal pour la non-existence de solution dans le demi-espace. \smallskip

Le th\'eor\`eme de Liouville pr\'ecedent peut etre appliqu\'e pour montrer \`a la non-existence de solution classique de l'\'equation fractionelle de  Lane-Emden
$$
 \displaystyle  (-\Delta)^\alpha   u = u^p\quad {\rm dans}\quad  \Omega,
\qquad  u\ge 0\quad {\rm dans}\quad \mathbb{ R}^N\setminus \Omega,
 $$
o\`u $\Omega=\mathbb{R}^N\setminus B_{r_0}(0)$ avec $r_0>0$ ou $\Omega=\mathbb{R}^{N-1}\times(0,+\infty)$.
\end{quote}

}

 \renewcommand{\thefootnote}{}
\footnote{*Corresponding author.}

\newcommand{\N}{\mathbb{N}}
\newcommand{\R}{\mathbb{R}}
\newcommand{\Z}{\mathbb{Z}}

\newcommand{\cA}{{\mathcal A}}
\newcommand{\cB}{{\mathcal B}}
\newcommand{\cC}{{\mathcal C}}
\newcommand{\cD}{{\mathcal D}}
\newcommand{\cE}{{\mathcal E}}
\newcommand{\cF}{{\mathcal F}}
\newcommand{\cG}{{\mathcal G}}
\newcommand{\cH}{{\mathcal H}}
\newcommand{\cI}{{\mathcal I}}
\newcommand{\cJ}{{\mathcal J}}
\newcommand{\cK}{{\mathcal K}}
\newcommand{\cL}{{\mathcal L}}
\newcommand{\cM}{{\mathcal M}}
\newcommand{\cN}{{\mathcal N}}
\newcommand{\cO}{{\mathcal O}}
\newcommand{\cP}{{\mathcal P}}
\newcommand{\cQ}{{\mathcal Q}}
\newcommand{\cR}{{\mathcal R}}
\newcommand{\cS}{{\mathcal S}}
\newcommand{\cT}{{\mathcal T}}
\newcommand{\cU}{{\mathcal U}}
\newcommand{\cV}{{\mathcal V}}
\newcommand{\cW}{{\mathcal W}}
\newcommand{\cX}{{\mathcal X}}
\newcommand{\cY}{{\mathcal Y}}
\newcommand{\cZ}{{\mathcal Z}}

\newcommand{\abs}[1]{\lvert#1\rvert}
\newcommand{\xabs}[1]{\left\lvert#1\right\rvert}
\newcommand{\norm}[1]{\lVert#1\rVert}

\newcommand{\loc}{\mathrm{loc}}
\newcommand{\p}{\partial}
\newcommand{\h}{\hskip 5mm}
\newcommand{\ti}{\widetilde}
\newcommand{\D}{\Delta}
\newcommand{\e}{\epsilon}
\newcommand{\bs}{\backslash}
\newcommand{\ep}{\emptyset}
\newcommand{\su}{\subset}
\newcommand{\ds}{\displaystyle}
\newcommand{\ld}{\lambda}
\newcommand{\vp}{\varphi}
\newcommand{\wpp}{W_0^{1,\ p}(\Omega)}
\newcommand{\ino}{\int_\Omega}
\newcommand{\bo}{\overline{\Omega}}
\newcommand{\ccc}{\cC_0^1(\bo)}
\newcommand{\iii}{\opint_{D_1}D_i}

\numberwithin{equation}{section}

\vskip 0.2cm \arraycolsep1.5pt
\newtheorem{lemma}{Lemma}[section]
\newtheorem{theorem}{Theorem}[section]
\newtheorem{definition}{Definition}[section]
\newtheorem{proposition}{Proposition}[section]
\newtheorem{remark}{Remark}[section]
\newtheorem{corollary}{Corollary}[section]

\setcounter{equation}{0}
\section{Introduction}

Let   $\Omega$ be a  $C^2$  domain in $\mathbb{R}^N$ satisfying that\\[1.5mm]
$(i)$  $\Omega\supseteq \left(\mathbb{R}^N\setminus \overline{B_{r_0}(0)}\right)$ \ \ or\ \ $(ii)$   $\Omega\supseteq \left\{x\in \mathbb{ R}^N:\, x\cdot a>0\right\}$,\\[1.5mm]
where   $r_0>0$ and $a\in \mathbb{ R}^N\setminus\{0\}$.
Our  purpose of this paper is to study the nonexistence of nonnegative very weak solutions to the fractional Lane-Emden type equation
\begin{equation}\label{eq 1.1}
\arraycolsep=1pt
\begin{array}{lll}
 \displaystyle
 \displaystyle  (-\Delta)^\alpha   u = u^p+\nu\quad
  &{\rm in}\quad  \Omega,\\[2mm]
 \phantom{(-\Delta)^\alpha }
 u=g\quad &{\rm in}\quad \R^N\setminus \Omega,
 \end{array}
\end{equation}
where    $p>0$, $\nu$ is a nonnegative Radon measure in $\Omega$, $g$ is a nonnegative function in $L^1(\R^N\setminus \Omega,\frac{dx}{1+|x|^{N+2\alpha}})$
 and  $(-\Delta)^\alpha$ with $\alpha\in(0,1)$    is   the fractional Laplacian defined in the  principle value sense,
$$(-\Delta)^\alpha  u(x)=c_{N,\alpha}\lim_{\epsilon\to0^+} \int_{\R^N\setminus B_\epsilon(0) }\frac{ u(x)-
u(x+z)}{|z|^{N+2\alpha}} \, dz,$$
here $B_\epsilon(0)$ is the ball with radius $\epsilon$ centered at the origin and  $c_{N,\alpha}>0$ is the normalized constant, see \cite{NPV} for details. In the particular case that
$\Omega=\R^N$ or $\Omega=\R^N\setminus \{x_0\}$ for some point $x_0\in\R^N$, the subjection: $u=g$ in $\R^N\setminus \Omega$ in (\ref{eq 1.1}) may be omitted.

It is known that the Liouville theorem  plays a crucial role in  deriving a priori estimates for solutions in PDE analysis and the nonexistence of entire solution to the second order differential equations has been  studied for some decades. There is a large literature on the nonexistence of solutions for the problem
 \begin{equation}\label{eq 1.04}
-\Delta u= f(u)\ \ {\rm in} \ \ \Omega,\qquad u\mid_{\partial\Omega}=0,\qquad \lim_{x\in\Omega,\,|x|\to+\infty} u(x)=0.
 \end{equation}
Berestycki and Lions in \cite{BL}  obtained the nonexistence results of (\ref{eq 1.04}) when $\Omega=\R^N$, Esteban \cite{E} made use of a version of Maximum Principle to study
the nonexistence of solutions of (\ref{eq 1.04}), when $\Omega$ is a strip tpye domain. For general unbounded domain, the nonexistence result was derived by Esteban and Lions  in \cite{EL}.
The Liouville theorem has been extended to the fully nonlinear elliptic equations, see the references \cite{AS1,BCN,BHR1,CC1,CF,CL,QS,LR},
by developing the basic tools:  the Maximum Principle and Hadamard Estimates.

During the last years,  there has been a renewed and increasing interest in the study of linear and nonlinear integral operators, especially, the fractional Laplacian,  motivated by great interest in the model diverse physical phenomena, such as anomalous diffusion and quasi-geostrophic flows, turbulence and water waves, molecular dynamics, and relativistic quantum mechanics of stars, see \cite{BG,CV,TZ} and by important advances on the theory of nonlinear partial differential equations. The Liouville theorem of the nonlocal elliptic problems has been attracting the attentions,  Felmer and Quaas in \cite{FQ2} extended the Hadamard estiamte for the fractional Pucci's operator and obtained the Liouville theorem for the corresponding Lane-Emden equations.

As a typical nolocal operator, the fractional Laplacian has been studied deeply,  Z. Chen et al in \cite{CS,CT} derived the estimates for its Green's kernels by the stochastic method,
W. Chen et al in \cite{CFY,CLL,C} obtained the nonexistence of the entire solution for the fractional  elliptic equations,   M. Fall and T. Weth in \cite{FM,FW} obtained the nonexistence of positive  solutions  for a class of fractional semilinear elliptic equations in unbounded domains.

Recently, H. Chen et al in \cite{CQ,CV1} studied the fractional elliptic equation with Radon measures in bounded domain. In particular, the fractional Lane-emden type equation
\begin{equation}\label{eq 1.1-1}
\arraycolsep=1pt
\begin{array}{lll}
 \displaystyle
 \displaystyle  (-\Delta)^\alpha   u = u^p+k\delta_0\quad
  &{\rm in}\quad  \Omega,\\[2mm]
 \phantom{(-\Delta)^\alpha }
 u=0\quad &{\rm in}\quad \R^N\setminus \Omega
 \end{array}
\end{equation}
has  very weak solution when $p<\frac{N}{N-2\alpha}$ and $k>0$ small, and has no very weak solution when $p\ge \frac{N}{N-2\alpha}$,  where $\Omega$ is a bounded domain containing the origin. One may ask if (\ref{eq 1.1-1}) has  very weak solution when the domain $\Omega$ is unbounded. Our motivation in this article is to clarify the existence and nonexistence when
 it involves unbounded domain, such as the whole domain, exterior domain and half space.

Before stating our main results, we make precise the notion of very weak solution used in this article.  {\it A function  $u$ is said to be a  very weak solution of (\ref{eq 1.1}) if $u\in L^1(\R^N, \frac{dx}{1+|x|^{N+2\alpha}})$,\ \  $|u|^p\in L^1_{loc}(\R^N)$ and
\begin{equation}\label{de 1.1}
  \int_{\Omega}u (-\Delta)^\alpha \xi\, dx=\int_{\Omega}u^p  \xi\, dx+\int_{\Omega}\xi\, d\nu-\int_{\Omega}  \xi (-\Delta)^\alpha  \tilde g \, dx,\quad \forall \xi\in C^\infty_c(\Omega),
\end{equation}
where $\tilde g=0$ is the zero extension of $g$ in $\Omega$,  $C^\infty_c(\Omega)$ is the space of all the functions in $C^\infty(\R^N)$ with compact support in $\Omega$. }  Although we set the the test function $\xi$ has
compact support in $\R^N$, it follows by the nonlocal property of the fractional Laplacian that $(-\Delta)^\alpha \xi (x)$ may have the decaying rate $|x|^{-N-2\alpha}$ as $|x|\to+\infty$. This decaying at infinity requires  that $u,\,g\in L^1(\R^N, \frac{dx}{1+|x|^{N+2\alpha}})$.

Now we state our first main results.

\begin{theorem}\label{teo 1}
Assume that $N>2\alpha$, $\nu$ is a nonnegative Radon measure in $\Omega$, $g$ is a nonnegative function in $L^1(\R^N\setminus \Omega,\frac{dx}{1+|x|^{N+2\alpha}})$.

$(i)$ Assume that  $\Omega\supseteq (\mathbb{R}^N\setminus \overline{B_{r_0}(0)})$ for some $r_0>0$,
\begin{equation}\label{1.1}
 p\in\left(0,\, \frac{N}{ N-2\alpha}\right).
\end{equation}
Then  for any nonnegative measure $\nu\not=0$, problem (\ref{eq 1.1}) has no nonnegative very weak solution. \smallskip

$(ii)$ Assume that  $\Omega\supseteq \{x\in \mathbb{ R}^N:\, x\cdot a>r_0\}$ for some $r_0>0$, $a\in \mathbb{ R}^N\setminus\{0\}$,
\begin{equation}\label{1.2}
 p\in\left(0,\, \frac{N+\alpha}{ N- \alpha}\right).
\end{equation}
Then for any nonnegative measure $\nu\not=0$, problem (\ref{eq 1.1}) has no nonnegative very weak solution. \smallskip
\end{theorem}

The nonexistence of very weak solution in Theorem \ref{teo 1} is derived by contradiction.  Let $u$ be a nonnegative very weak solution of problem (\ref{eq 1.1}),  the nonnegative source $\nu$ would provide $u$ an initial positive decay at infinity, then this decay will be reacted by the source nonlinearity $u^p$, until  that $u$  blows up  everywhere.   To our knowledge, our method is new and it could be applied  in the Laplacian case, since it requires only the basic tools: the comparison principles (or Kato's inequality) and  the estimates of the corresponding Green's kernel.    This method, of course, is  suitable in the classical sense.
 We say that  $u$ is a classical  solution of
  \begin{equation}\label{eq 3.0}
  \arraycolsep=1pt
\begin{array}{lll}
 \displaystyle   (-\Delta)^\alpha   u = u^p\quad
 & {\rm in}\quad  \Omega,\\[2mm]
 \phantom{  (-\Delta)^\alpha }
\displaystyle    u =g\quad
 & {\rm in}\quad  \R^N\setminus\Omega,
\end{array}
  \end{equation}
if $u$ is continuous in $\Omega$ and satisfies the first equation in (\ref{eq 3.0}) pointwise in  $\Omega$, where the function $g$ is an outside source.

When $\Omega$ is an exterior domain,
i.e. $\Omega=\R^N\setminus \overline{B_{r_0}(0)}$, we have the following nonexistence of results.
\begin{corollary}\label{cr 1}
Assume that $p\in(0,\, \frac{N}{ N-2\alpha})$   and $u$ is a nonnegative classical solution of  problem
\begin{equation}\label{eq 1.2}
  \arraycolsep=1pt
\begin{array}{lll}
 \displaystyle   (-\Delta)^\alpha   u =  u^p\quad
 & {\rm in}\quad  \R^N\setminus B_{r_0}(0),\\[2mm]
 \phantom{  (-\Delta)^\alpha }
\displaystyle    u \ge 0\quad
 & {\rm a.e.\ in}\quad    B_{r_0}(0).
\end{array}
\end{equation}
 Then $u\equiv0$ a.e. in $\R^N$.
\end{corollary}

In the half space case, i.e. $\Omega= \R^{N-1}\times (0,\infty)$, we have the following corollary.

\begin{corollary}\label{cr 2}
Assume that  $p\in(0,\, \frac{N+\alpha}{ N-\alpha})$  and $u$ is a nonnegative classical solution of  problem
\begin{equation}\label{eq 1.3}
  \arraycolsep=1pt
\begin{array}{lll}
 \displaystyle   (-\Delta)^\alpha   u =  u^p\quad
 & {\rm in}\quad  \R^N_+,\\[2mm]
 \phantom{  (-\Delta)^\alpha }
\displaystyle    u \ge 0\quad
 & {\rm a.e.\ in}\quad    \R^{N}\setminus \R^N_+,
\end{array}
\end{equation}
where $ \R^N_+=\R^{N-1}\times (0,\infty)$.
Then $u\equiv0$ a.e. in $\R^N$.
\end{corollary}

Turning back to Theorem \ref{teo 1}, we  notice that in the case  $\Omega\supseteq \R^N\setminus \overline{B_{r_0}(0)}$,  problem (\ref{eq 1.1}) has no very weak solution when
$\nu=\delta_{\bar x}$ and $p\ge \frac{N}{N-2\alpha}$,  because of the singularity at the point $\bar x\in\Omega$, here  $\frac{N}{N-2\alpha}$ is called by Serrin type exponent in the exterior domain or the whole domain. In the half space case $\Omega=\R^N_+:=\R^{N-1}\times(0,+\infty)$,
the Serrin type exponent is $\frac{N+\alpha}{ N-\alpha}$, which is sharp for the nonexistence.  In fact, problem (\ref{eq 1.1}) has a very weak solution under the hypotheses
that $\nu$ is a Dirac mass and $\frac{N+\alpha}{ N-\alpha}\le p< \frac{N}{N-2\alpha}$.
Precisely, the existence result reads as follows.

\begin{theorem}\label{teo 2}
Assume that $k>0$, $\delta_{e_N}$ is the Dirac mass concentrated on the point $e_N=(0,\cdots,0,1)$ and
\begin{equation}\label{6.1}
p\in\left[\frac{N+\alpha}{N-\alpha}, \frac{N}{N-2\alpha}\right).
\end{equation}
Then there exists $k^*>0$ such that for
$k\in (0,k^*)$,
\begin{equation}\label{eq 6.1}
  \arraycolsep=1pt
\begin{array}{lll}
 \displaystyle   (-\Delta)^\alpha   u =  u^p+k\delta_{e_N}\quad
 & {\rm in}\quad  \R^{N}_+,\\[2mm]
 \phantom{  (-\Delta)^\alpha }
\displaystyle    u =0\quad
 & {\rm in}\quad    \R^N\setminus \R^{N}_+
\end{array}
\end{equation}
 admits a minimal positive solution $u_k$, which is a classical solution of
 \begin{equation}\label{eq 1.5}
  \arraycolsep=1pt
\begin{array}{lll}
 \displaystyle   (-\Delta)^\alpha   u =  u^p \quad
 & {\rm in}\quad  \R^{N}_+\setminus\{e_N\},\\[2mm]
 \phantom{  (-\Delta)^\alpha }
\displaystyle    u =0\quad
 & {\rm in}\quad    \R^N\setminus \R^{N}_+.
\end{array}
\end{equation}

\end{theorem}

In  contrast with the existence of positive solutions to (\ref{eq 1.5}), Chen, Fang and Yang in \cite{CFY} obtained that the problem
\begin{equation}\label{eq 1.4}
  \arraycolsep=1pt
\begin{array}{lll}
 \displaystyle   (-\Delta)^\alpha   u =  u^p\quad
 & {\rm in}\quad  \R^N_+,\\[2mm]
 \phantom{  (-\Delta)^\alpha }
\displaystyle    u = 0\quad
 & {\rm  \ in}\quad   \R^N\setminus \R^N_+
\end{array}
\end{equation}
has only zero nonnegative solution  under the hypotheses that
$$p>\frac{N}{N-2\alpha}\quad{\rm and}\quad u\in L^{\frac{N(p-1)}{2\alpha}}(\R^N_+).$$
When $1<p<\frac{N-1+2\alpha}{N-1-2\alpha}$, the nonexistence of positive bounded classical solution to (\ref{eq 1.4}) has been obtained independently in  \cite{FM,QX}.  These Liouville type theorems are derived by employing the method of moving planes. However, when it involves nontrivial nonnegative outside source, the method of moving planes is no longer valid and the critical exponent for the nonexistence reduces to
 $\frac{N+\alpha}{N-\alpha}$, see Corollary \ref{cr 2},  and Theorem \ref{teo 2} provides an example showing the existence when  $p\ge \frac{N+\alpha}{N-\alpha}$.
 In fact, let  $w(x)=u_k(x+2e_N)$, where $u_k$ is the very weak solution of (\ref{eq 6.1}),  then $w$ is a classical solution of (\ref{eq 1.3}) with nontrivial nonnegative outside source.

The paper is organized as follows. In Section 2, we  show basic properties of the solutions  to nonhomogeneous problem with nonzero outside source,   the Integration by
Parts formula,  Comparison Principle.     Section  3 is devoted
to prove the nonexistence   of nonnegative solutions to (\ref{eq 1.1}) and to prove the nonexistence in the classical setting.
 Finally, we prove the existence  of very weak solutions
 of (\ref{eq 1.2}).

\section{Preliminary}

Given a $C^2$ domain $\mathcal{O}$,  denote $ d_{\mathcal{O}}(x)={\rm dist}(x,\partial \mathcal{O})$, denote by $G_{\alpha,\mathcal{O}}$ the Green's function in $\mathcal{O}\times \mathcal{O}$ and by $\mathbb{G}_{\alpha,\mathcal{O}}[\nu]$   the very weak solution of
$$
\arraycolsep=1pt
\begin{array}{lll}
(-\Delta)^\alpha &u=\nu\quad &  {\rm in}\quad\mathcal{O},\\[2mm]
 &u=0\quad & {\rm in}\quad \R^N\setminus\mathcal{O},
\end{array}
$$
where $\nu$ is a Radon measure in $\mathcal{O}$.
 In fact, for almost every $x\in\mathcal{O}$,
 $$\mathbb{G}_{\alpha,\mathcal{O}}[\nu](x)=\int_{\mathcal{O}}G_{\alpha,\mathcal{O}}(x,y)d\nu(y).$$
In what follows, we always denote by $c_i$  the positive constant with $i\in\N$.

We first introduce the strong Comparison Principle.

\begin{lemma}\label{lm 2.0}
Assume that $\mathcal{D}$ is a $C^2$ domain,
functions $f_1,\,f_2\in C(\mathcal{D})$  satisfy
 $f_2\ge f_1$ in $\mathcal{D}$ and $g_1,\,g_2\in C(\mathbb{R}^N\setminus \mathcal{D})$  satisfy
 $g_2\ge g_1$.

Let $u_i$ be the classical solutions of
$$
\arraycolsep=1pt
\begin{array}{lll}
(-\Delta)^\alpha &u=f_i\quad & {\rm in}\quad\mathcal{D},\\[2mm]
 &u=g_i\quad & {\rm in}\quad \R^N\setminus\mathcal{D}
\end{array}
$$
with $i=1,2$, respectively.
If
$$\liminf_{x\to \partial \mathcal{D}}u_2(x)\ge\limsup_{x\to \partial \mathcal{D}}u_1(x) $$
and for unbounded domain,
$$\liminf_{|x|\to\infty}u_2(x)\ge\limsup_{|x|\to \infty }u_1(x), $$
then
$$u_2\ge u_1\quad {\rm in}\ \ \mathcal{D}.$$

\end{lemma}
\begin{proof}
If $\inf_{x\in \mathcal{D}}(u_2-u_1)(x)<0$, then  there exists a point $x_0\in \mathcal{D}$ such that
$(u_2 -u_1)(x_0)=\inf_{x\in \mathcal{D}}(u_2-u_1)(x)={\rm essinf}_{x\in \R^N}(u_2-u_1)(x),$
which implies that
$$(-\Delta)^\alpha (u_2-u_1)(x_0)={\rm P.V.} \int_{\R^N}\frac{(u_2-u_1)(x_0)-(u_2-u_1)(y)}{|x_0-y|^{N+2\alpha}}\,dy<0.$$
However,
$$(-\Delta)^\alpha (u_2-u_1)(x_0)=f_2(x_0)-f_1(x_0)\ge 0,$$
which is impossible.
\end{proof}

Next we introduce the weak Comparison Principles, which could be derived by the Kato's inequality in the fractional setting.
\begin{proposition}\label{pr 2.1}\cite[Proposition 2.4]{CV1}
Assume that $\mathcal{O}$ is a $C^2$ bounded domain  and $f\in L^1(\mathcal{O}, d_\mathcal{O}^\alpha dx)$, where $d_\mathcal{O}(x)={\rm dist}(x,\partial\mathcal{O})$. Then
there exists a  unique weak
solution $u$ to the problem
\begin{equation}\label{homo}
 \arraycolsep=1pt
\begin{array}{lll}
 (-\Delta)^\alpha &u=f\quad & {\rm in}\quad\mathcal{O},\\[2mm]
&u=0\quad & {\rm in}\quad \R^N\setminus\mathcal{O}.
\end{array}
\end{equation}
Moreover, for any $\xi\in C^{1.1}(\mathcal{O})\cap C^\alpha_0(\mathcal{O})$, $\xi\ge0$, we have that
 \begin{equation}\label{sign}
\int_{\mathcal{O}} |u|(-\Delta)^\alpha \xi dx\le \int_{\mathcal{O}}  \xi\
{{\rm sign}}(u)f \ dx
\end{equation}
and
 \begin{equation}\label{sign+}
\int_{\mathcal{O}} u_+(-\Delta)^\alpha \xi dx\le \int_{\mathcal{O}}  \xi \
{{\rm sign}}_+(u)f \ dx.
\end{equation}

\end{proposition}

 We say that $u_g$ is a very weak solution of
 \begin{equation}\label{eq 2.5}
\arraycolsep=1pt
\begin{array}{lll}
 \displaystyle
 \displaystyle  (-\Delta)^\alpha   u = \nu\quad
  &{\rm in}\quad  \mathcal{O},\\[2mm]
 \phantom{(-\Delta)^\alpha }
 u=g\quad &{\rm in}\quad \R^N\setminus \mathcal{O},
 \end{array}
\end{equation}
if $u\in L^1(\mathcal{O}, d_{\mathcal{O}}(x)^\alpha dx)$ and
\begin{equation}\label{e 2.1}
 \int_\mathcal{O} u_g (-\Delta)^\alpha \xi dx=\int_\mathcal{O} \xi d\nu -\int_\mathcal{O}  \xi (-\Delta)^\alpha  \tilde g  dx,\qquad \forall \xi\in C^\infty_c(\mathcal{O}),
\end{equation}
where $\nu$ is a bounded Radon measure in $\mathcal{O}$, $g\in L^1(\mathcal{O}^c,\frac{dx}{1+|x|^{N+2\alpha}})$,  $\tilde g$ is the zero extension of $g$ in $\mathcal{O}$. When $g=0$ in $\R^N\setminus \mathcal{O}$, The authors in\cite{CV1} showed that problem (\ref{eq 2.5})
admits a unique very weak solution $u_0$. When it involves the nonzero outside source $g$, the first difficulty is  the Integration by Parts formula.
 From the definition of $(-\Delta)^\alpha$, we have that
\begin{equation}\label{2.0}
 (-\Delta)^\alpha  \tilde g(x)=c_{N,\alpha}\int_{\R^N\setminus \mathcal{O}}\frac{g(y)}{|x-y|^{N+2\alpha}}\,dy,\qquad \forall x\in\mathcal{O}.
\end{equation}

\begin{lemma}\label{cpg}
Assume that $\mathcal{O}$ is a bounded $C^2$ domain,  $\nu$ is a bounded Radon measure, $g\in L^1(\mathcal{O}^c,\frac{dx}{1+|x|^{N-2\alpha}})$
and $u_0$ is a very weak solution of (\ref{eq 2.5}) with $g=0$ in $\R^N\setminus \mathcal{O}$.

Then   (\ref{eq 2.5}) admits  a unique very weak solution $u_g$ satisfying that
$${\rm if}\ \ g\ge 0,\ \  {\rm then}\ \ u_g\ge u_0$$
and
$${\rm if}\ \ g\le 0,\ \  {\rm then}\ \ u_g\le u_0.$$
\end{lemma}
\begin{proof}
Let  $\nu_n$ be a $C^2$ sequence of functions converging to $\nu$ in the dual sense of $C(\bar{\mathcal{O}})$.
For $\xi\in C^\infty_c(\mathcal{O})$, there exists $r_1>0$ such that
$${\rm supp}\,\xi\subset \mathcal{O}_1:=\{x\in \mathcal{O}:\, {\rm dist}(x,\partial\mathcal{O})> r_1\}.$$
Let $u_n$ be the classical solution of
 \begin{equation}\label{eq 2.4}
\arraycolsep=1pt
\begin{array}{lll}
 \displaystyle
 \displaystyle  (-\Delta)^\alpha   u = \nu_n\quad
  &{\rm in}\quad  \mathcal{O}_1,\\[2mm]
 \phantom{(-\Delta)^\alpha }
 u=\tilde g\quad &{\rm in}\quad \R^N\setminus \mathcal{O}_1.
 \end{array}
\end{equation}
Since $\tilde g=0$ in $\mathcal{O}$, then
$(-\Delta)^\alpha \tilde g\in C^1_{loc}(\mathcal{O})$. Let
  $w_n=u_n-\tilde g$ and $w_n$ is a classical solution of
 $$
 \arraycolsep=1pt
\begin{array}{lll}
 \displaystyle
 \displaystyle  (-\Delta)^\alpha   u = \nu_n-(-\Delta)^\alpha\tilde g\quad
  &{\rm in}\quad  \mathcal{O}_1,\\[2mm]
 \phantom{(-\Delta)^\alpha }
 u=0\quad &{\rm in}\quad \R^N\setminus \mathcal{O}_1.
 \end{array}
 $$
 Let $v_n$ be the classical solution of
 $$
 \arraycolsep=1pt
\begin{array}{lll}
 \displaystyle
 \displaystyle  (-\Delta)^\alpha   u = \nu_n\quad
  &{\rm in}\quad  \mathcal{O}_1,\\[2mm]
 \phantom{(-\Delta)^\alpha }
 u=0\quad &{\rm in}\quad \R^N\setminus \mathcal{O}_1.
 \end{array}
 $$

If $g\ge0$, it follows by (\ref{2.0}) that $(-\Delta)^\alpha \tilde g\ge 0$, and by Comparison Principle, we have that
 \begin{equation}\label{e 2.2}
 v_n\le u_n\quad{\rm in}\quad \R^N.
 \end{equation}
By the Integration by Parts formula, see \cite[Lemma 2.2]{CV1},
we know that
\begin{equation}\label{e 2.3}
\int_{\mathcal{O} } w_n (-\Delta)^\alpha \xi dx=\int_{\mathcal{O} } \xi \nu_ndx- \int_{\mathcal{O} } \xi (-\Delta)^\alpha  \tilde g dx.
\end{equation}
From   \cite[Proposition 2.2]{CV1}, it holds that
\begin{eqnarray*}
\norm{w_n}_{M^\frac{N}{N-2\alpha}(\mathcal{O}_1)} &\le & c_1 \norm{\nu_n}_{L^1(\mathcal{O}_1)}+ c_1\norm{(-\Delta)^\alpha  \tilde g }_{L^1(\mathcal{O}_1)} \\
   &\le & c_2 \norm{\nu}_{\mathfrak{M}(\mathcal{O})}+ c_2\norm{(-\Delta)^\alpha  \tilde g }_{L^1(\mathcal{O}_1)}
\end{eqnarray*}
where  $M^\frac{N}{N-2\alpha}(\mathcal{O}_1)$ is the Marcinkiewicz  space with exponent $\frac{N}{N-2\alpha}$ in $\mathcal{O}_1$ and $\mathfrak{M}(\mathcal{O})$
is the bounded Radon measure space in $\mathcal{O}$.
By  \cite[Proposition 2.6]{CV1},  there exists $u_g\in L^1(\mathcal{O})$ such that, up to some subsequence,
$$v_n\to u_0,\ \ w_n\to u_g\quad{\rm in}\quad L^1(\mathcal{O})\quad{\rm as}\ \ n\to+\infty.$$
Passing the limit in (\ref{e 2.3})  as $n\to+\infty$, we have that
$$\int_{\mathcal{O}} u_g (-\Delta)^\alpha \xi\, dx=\int_{\mathcal{O} } \xi\, d \nu - \int_{\mathcal{O} } \xi (-\Delta)^\alpha  \tilde g \,dx. $$
It follows by (\ref{e 2.2})  that
$$u_g\ge u_0\quad{\rm in}\quad\R^N.$$

Now we prove  the uniqueness.  Assume that problem (\ref{eq 2.5}) has two solutions $u_1,\,u_2$, then $w=u_1-u_2$ is a very weak solution of
 $$\arraycolsep=1pt
\begin{array}{lll}
 (-\Delta)^\alpha &u=0\quad & {\rm in}\quad\mathcal{O},\\[2mm]
&u=0\quad & {\rm in}\quad \R^N\setminus\mathcal{O}.
\end{array}
$$
Then nonhomogeneous term is zero, of course, which is in $L^1(\mathcal{O})$, so by applying Proposition \ref{pr 2.1}, we have that $w\equiv 0$ a.e. in $\mathcal{O}$.
The proof is complete. \end{proof}

\begin{remark}
From Divergence theorem, the following identity holds
$$\int_{\mathcal{O}}(-\Delta)  \xi(x)\, dx=0,\quad \forall\, \xi\in C_c^\infty(\mathcal{O}).$$
 In contrast with the Laplacian case, the corresponding identity for the fractional Laplacian reads
$$\int_{\mathcal{O}}(-\Delta)^\alpha \xi(x)\, dx=\int_{\mathcal{O}} \xi (x)\int_{\R^N\setminus \mathcal{O}}\frac{c_{N,\alpha}}{|x-y|^{N+2\alpha}}\, dy dx,\quad \forall\, \xi\in C_c^\infty(\mathcal{O}),$$
which could be obtained from (\ref{e 2.1}) by the solution  $u\equiv1$  of (\ref{eq 2.5}) taking $\nu=0$ and $g=1$ in $\R^N\setminus \mathcal{O}$.
\end{remark}

\begin{corollary}\label{cr 3}
Let $\mathcal{O}_1$ and $\mathcal{O}_2$ be   $C^2$ domains  such that
$$\mathcal{O}_1\subset \mathcal{O}_2.$$
Then
\begin{equation}\label{eq 2.6}
G_{\alpha,\mathcal{O}_1}(x,y) \le G_{\alpha,\mathcal{O}_2}(x,y), \quad \forall\, (x,y)\in \R^N\times\R^N,\ x\not=y.
\end{equation}

\end{corollary}
\begin{proof}
\emph{Case 1:} $\mathcal{O}_1$ is bounded. Since $\mathcal{O}_1\subset \mathcal{O}_2$, then for fixed $y\in \mathcal{O}_1$,
$G_{\alpha,\mathcal{O}_1}(\cdot,y), G_{\alpha,\mathcal{O}_2}(\cdot,y)$ are the solutions to
$$(-\Delta)^\alpha u=\delta_y\quad{\rm in}\ \ \mathcal{O}_1,$$
subjecting to
$u=0$ in $\R^N\setminus \mathcal{O}_1$ and to $u=G_{\alpha,\mathcal{O}_2}(\cdot,y)$ in $\R^N\setminus \mathcal{O}_1$, respectively.
Then applying Lemma \ref{cpg}, we obtain that
$$G_{\alpha,\mathcal{O}_1}(x,y) \le G_{\alpha,\mathcal{O}_2}(x,y), \quad \ \forall (x,y)\in \R^N\times\R^N,\ x\not=y.$$
For fixed  $y\in \R^N\setminus\mathcal{O}_1$, obvious that $G_{\alpha,\mathcal{O}_1}(\cdot,y)\equiv 0$ in $\R^N$ and
$G_{\alpha,\mathcal{O}_2}(\cdot,y)\ge 0$ in $\R^N\setminus\{y\}$.
Thus, (\ref{eq 2.6}) holds.

\smallskip

Case 2: $\mathcal{O}_1$ is unbounded.  For any $y\in\mathcal{O}_1$, let $w=G_{\alpha,\mathcal{O}_1}(\cdot,y)-G_{\alpha,\mathcal{O}_2}(\cdot,y)$,
then $w$ is a classical solution of
$$
 \arraycolsep=1pt
\begin{array}{lll}
 \displaystyle
 \displaystyle  (-\Delta)^\alpha   w= 0\quad
  &{\rm in}\quad  \mathcal{O}_1\setminus\{y\},\\[2mm]
 \phantom{(-\Delta)^\alpha }
 w\le 0\quad &{\rm in}\quad \R^N\setminus \mathcal{O}_1.
 \end{array}
 $$
By the basic facts
$$\lim_{x\to y}G_{\alpha,\mathcal{O}_1}(x,y)|x-y|^{N-2\alpha}=\lim_{x\to y}G_{\alpha,\mathcal{O}_2}(x,y)|x-y|^{N-2\alpha}=c_3, $$
 we derive that
$$\lim_{x\to y}w(x)|x-y|^{N-2\alpha}=0,$$
thus, for any $\epsilon$, there exists $r>0$ such that
 $w(x)\le \epsilon G_{\alpha,\R^N}(x,y), \ \forall x\in B_\epsilon(0)\setminus\{y\}$.
 We observe that
\begin{equation}\label{a 1}
 (-\Delta)^\alpha\left(w-\epsilon G_{\alpha,\R^N}(\cdot,y)\right)=0\quad{\rm in}\quad \mathcal{O}_1\setminus B_\epsilon(0).
\end{equation}
Since $G_{\alpha,\R^N}(\cdot,y)>0$ in $\R^N\setminus \mathcal{O}_1$
and
$$\lim_{|x|\to+\infty}\epsilon G_{\alpha,\R^N}(\cdot,y)=0=\lim_{|x|\to+\infty}w(x) $$
so if
$$\sup_{x\in\R^N\setminus\{y\}}\left(w-\epsilon G_{\alpha,\R^N}(\cdot,y)\right)>0,$$
there exists $x_0\in \mathcal{O}_1\setminus B_\epsilon(0)$ such that
$$w(x_0)-\epsilon G_{\alpha,\R^N}(x_0,y)=\sup_{x\in\R^N\setminus\{y\}}\left(w-\epsilon G_{\alpha,\R^N}(\cdot,y)\right)>0,$$
and from the definition of the fractional Laplacian, we obtain that
\begin{eqnarray*}
&&(-\Delta)^\alpha\left(w-\epsilon G_{\alpha,\R^N}(\cdot,y)\right)(x_0)   \\
    &&=c_{N,\alpha}\int_{\R^N}\frac{\left(w(x_0)-\epsilon G_{\alpha,\R^N}(x_0,z)\right) -\left(w(z)-\epsilon G_{\alpha,\R^N}(z,y)\right) }{|x_0-z|^{N-2\alpha}}dz>0,
\end{eqnarray*}
which contradicts (\ref{a 1}). This is to say that for any $\epsilon>0$,
$$w\le \epsilon G_{\alpha,\R^N}(\cdot,y)\quad{\rm in}\quad\R^N\setminus\{y\}.$$
Therefore,  $w\le 0$ in $\R^N\setminus\{y\}$ and (\ref{eq 2.6}) holds.
The proof is complete. \end{proof}

Next we make a general estimate of the very weak solution of fractional equation with nonlinearity.
\begin{lemma}\label{cp}
Assume that $\mathcal{D}$ is a $C^2$ domain in $\R^N$ (not necessary bounded),  $f:[0,+\infty)\to [0,+\infty)$
and $\nu$ is a nonnegative Radon measure satisfying
\begin{equation}\label{1.3}
 \int_{\Omega}\frac{1}{1+|y|^{N-2\alpha}}d\nu<+\infty.
\end{equation}
Let $u\ge0$ be a very weak solution of
\begin{equation}\label{eq 2.2}
\arraycolsep=1pt
\begin{array}{lll}
 \displaystyle
 \displaystyle  (-\Delta)^\alpha   u =  f(u)+\nu\quad
  &{\rm in}\quad  \mathcal{D},\\[2mm]
 \phantom{(-\Delta)^\alpha }
 u\ge0\quad &{\rm in}\quad \R^N\setminus \mathcal{D}.
 \end{array}
\end{equation}
Then
$$u\ge  \mathbb{G}_{\alpha,\mathcal{D}}[\nu]\quad{\rm a.e.\ in}\ \mathcal{D}.$$
\end{lemma}
\begin{proof}
By (\ref{1.3}), we observe that $ \mathbb{G}_{\mathcal{D},\alpha}[\nu]$ is well defined. Let $D_n$ be a $C^2$ bounded domain such that
$$  \mathcal{D}\cap B_{n}(0) \subset  D_n\subset \mathcal{D}\cap B_{n+1}(0)\quad{\rm and}\quad \mathcal{D}=\bigcup_{n=1}D_n.$$
Let $u_n$ be the positive very weak solution of
$$
\arraycolsep=1pt
\begin{array}{lll}
(-\Delta)^\alpha &u_n=u^p+\nu\quad & {\rm in}\quad D_n,\\[2mm]
 &u_n=0\quad & {\rm in}\quad \R^N\setminus D_n.
\end{array}
$$
Since $u\ge0$, then it follows by Corollary \ref{cr 3} that for any $n$,
\begin{equation}\label{2.02}
u \ge  u_n\quad{\rm a.e.\ in}\ \ \R^N.
\end{equation}

Let
$v_n$ be the positive very weak solution of
$$
\arraycolsep=1pt
\begin{array}{lll}
(-\Delta)^\alpha &v_n=\nu\quad & {\rm in}\quad D_n,\\[2mm]
 &v_n=0\quad & {\rm in}\quad \R^N\setminus D_n.
\end{array}
$$
Then $w_n:=v_n-u_n$ is a weak solution of
$$
\arraycolsep=1pt
\begin{array}{lll}
(-\Delta)^\alpha &w_n=-u^p\quad & {\rm in}\quad D_n,\\[2mm]
 &w_n=0\quad & {\rm in}\quad\R^N\setminus D_n.
\end{array}
$$
From the Kato's inequality (\ref{sign+}) with $\xi$ the first eigenfunction of $((-\Delta)^\alpha, D_n)$, we have that
$$w_n\le 0\quad{a.e.\ in}\ D_n,$$
thus, together with (\ref{2.02}), we have that
\begin{equation}\label{2.5}
 u\ge v_n\quad{\rm a.e.\ in}\ \ \R^N.
\end{equation}
Passing the limit as $n\to\infty$, we have that
$$u\ge  \mathbb{G}_{\alpha,\mathcal{D}}[\nu].$$
The proof is complete.
\end{proof}

\begin{remark}\label{rm 2}
Under the hypothesises of Lemma \ref{cp}, let $\mu$ be a nonnegative Radon measure
such that
$$\mu\le \nu.$$
Then the nonnegative weak solution $u$ of (\ref{eq 2.2}) satisfies
$$u\ge  \mathbb{G}_{\alpha,\mathcal{D}}[\mu]\quad{\rm a.e.\ in}\ \mathcal{D}.$$

\end{remark}

 Let $\tau_0<0$ and   $\{\tau_j\}_j$ be the sequence generated by
\begin{equation}\label{2.1}
 \tau_j=2\alpha+p\tau_{j-1}\quad{\rm for}\quad  j=1,2,3\cdots.
\end{equation}
\begin{lemma}\label{lm 2.1}
Assume that
\begin{equation}\label{2.00}
 p\in\left(0,\ 1+\frac{2\alpha}{-\tau_0}\right),
\end{equation}
  then $\{\tau_j\}_j$ is an increasing sequence and
there exists $j_0\in\N $  such that
\begin{equation}\label{2.3}
 \tau_{j_0}\ge0\quad {\rm and}\quad \tau_{j_0-1}<0.
\end{equation}
\end{lemma}
\begin{proof}
For $p\in(0,1+\frac{2\alpha}{-\tau_0})$, we have that
$$\tau_1-\tau_0=2\alpha+\tau_0(p-1)>0$$
and
\begin{equation}\label{2.2}
\tau_j-\tau_{j-1} = p(\tau_{j-1}-\tau_{j-2})=p^{j-1} (\tau_1-\tau_0),
\end{equation}
which imply that the sequence $\{\tau_j\}_j$ is  increasing.
If $p\ge1 $, our conclusions are obvious. If $p\in(0,1)$,
we have that in the case that $\tau_1\ge0$, we are done, and in the case that $\tau_1<0$,
it deduces from (\ref{2.2}) that
\begin{eqnarray*}
\tau_j  &=&  \frac{1-p^j}{1-p}(\tau_1-\tau_0)+\tau_0\\
&\to&\frac{1}{1-p}(\tau_1-\tau_0)+\tau_0=\frac{2\alpha}{1-p}>0,\quad {\rm as}\quad j\to+\infty,
\end{eqnarray*}
 then there exists $j_0>0$ satisfying (\ref{2.3}).
\end{proof}
In the section 3, we shall apply lemma \ref{lm 2.1}  with   $\tau_0=2\alpha-N$
when $\Omega\supset (\R^N\setminus B_{r_0}(0))$,  with
$\tau_0=\alpha-N$ when $\Omega\supseteq \left\{x\in \mathbb{ R}^N:\, x\cdot a>0\right\}$. Furthermore, for $\tau_0\in(-N,-N+2\alpha]$, it deduces by the fact
$\tau_{j_0-1}<0$ that if $j_0\ge 2$,
$$
\tau_0+p\tau_{j_0-2}<-N.
$$

\section{ Nonexistence }

We prove the nonexistence of very weak solution of (\ref{eq 1.1}) by contradiction. Assume that problem (\ref{eq 1.1})
admits a  nonnegative solution $u$, we will obtain a contradiction from its decay at infinity.

\subsection{The whole domain or the exterior domain}

In the case that $\Omega=\R^N$, the Green's function is
$$G_{\alpha,\Omega}(x,y)=\frac{c_3}{|x-y|^{N-2\alpha}},\quad\forall\, x,y\in\R^N,\ x\not=y.$$

For the general exterior domain, the Green's kernel couldn't be given explicitly, but we can give the following
estimate, which will play an important role in the derivation of the decay at infinity of the nonnegative solution $u$ to  problem (\ref{eq 1.1}).
\begin{lemma}\label{lm 2.3}
Assume that  $N>2\alpha$, $\Omega\supseteq (\mathbb{R}^N\setminus B_{r_0}(0))$.
Let $G_{\alpha,\Omega}$ be the Green's function of $(-\Delta)^\alpha$ in $\Omega\times\Omega$,
then there exists $c_4>0$ such that for $x,y\in \R^N\setminus B_{4r_0}(0)$,
 \begin{equation}\label{3.2}
 G_{\alpha,\Omega}(x,y)\ge \frac{c_4}{|x-y|^{N-2\alpha}}.
 \end{equation}

\end{lemma}
\begin{proof}
 From the scaling property, see \cite[(1.2)]{CS}, for any $l>0$ and any bounded $C^{1,1}$ domain $\mathcal{O}$, there holds
 \begin{equation}\label{2.9}
G_{\alpha,\mathcal{O}}(x,y)=l^{2\alpha-N} G_{\alpha,l\mathcal{O}}(\frac xl,\frac yl).
 \end{equation}
 We may assume that $r_0=\frac12$.   Fixed  $y\in \R^N\setminus B_2(0)$, let $\Gamma_y$ be the solution of
\begin{equation}\label{3.3}
\arraycolsep=1pt
\begin{array}{lll}
 \displaystyle (-\Delta)^\alpha  u=\delta_y-|y|^{2\alpha-N}\delta_0\quad & {\rm in}\quad \R^N,\\[2mm]
 \displaystyle \lim_{|x|\to+\infty}u(x)=0.
\end{array}
\end{equation}
Then we have that
\begin{equation}\label{3.7}
 \Gamma_y(x)=\frac{c_3}{|x-y|^{N-2\alpha}}-\frac{c_3|y|^{2\alpha-N}}{|x|^{N-2\alpha}},\quad\forall\, x\in\R^N\setminus\{y,0\}.
\end{equation}
Denote
$$A_y=\left\{x\in\R^N\setminus\{y,0\}: \Gamma_y(x)\le 0\right\}$$
and $x\in A_y$ if and only if
$$|y-x|\ge|y||x|.$$

On the one hand, for $x$ satisfying
\begin{equation}\label{3.4}
 |y|-|x|\ge |y||x|,
\end{equation}
we have that $x\in A_y$. We observe that
 (\ref{3.4}) is equivalent to
 $$|x|\le \frac{|y|}{|y|+1},$$
that is,
$$B_{\frac{|y|}{|y|+1}}(0)\subset A_y,$$
which implies that for any $|y|\ge 2$,
$$B_{r_0}(0)\subset A_y.$$

On the other hand, for any $x\in A_y$, we have that
$$ |y|+|x|\ge |y||x|,$$
that is,
$$|x|\ge \frac{|y|}{|y|-1}.$$
So  for any $|y|\ge 2$,
$$ A_y \subset B_{\frac{|y|}{|y|-1}}(0)\subset B_2(0).$$

We see that $G_{\alpha, \Omega}(\cdot,y)$ is the very weak solution of
\begin{equation}\label{3.5}
\arraycolsep=1pt
\begin{array}{lll}
 \displaystyle (-\Delta)^\alpha  u=\delta_y\quad & {\rm in}\quad \Omega,\\[2mm]
  \phantom{(-\Delta)^\alpha }
 u=0\quad& {\rm in}\quad \R^N\setminus\Omega,\\[2mm]
 \displaystyle \lim_{|x|\to+\infty}u(x)=0.
\end{array}
\end{equation}
Since $\R^N\setminus\Omega\subset B_{r_0}(0)$,  it follows by Corollary \ref{cr 3}  that
\begin{equation}\label{3.6}
 G_{\alpha, \Omega}(\cdot,y)\ge \Gamma_y\quad {\rm in}\quad \R^N\setminus\{y\}.
\end{equation}
It follows by (\ref{3.7}) that for $|x|\ge 2$,
$$|y||x|\ge  \frac{|x|+|y|}2 \ge \frac{|x-y|}2,$$
which implies that
$$\Gamma_y(x)\ge \left(1-2^{2\alpha-N}\right) \frac{c_{N,\alpha}}{|x-y|^{N-2\alpha}}.$$
Thus, (\ref{3.2}) holds. The proof is complete.
\end{proof}

\begin{lemma}\label{lm 3.1}
Let   $\{\tau_j\}_j$ be defined by (\ref{2.1}) with $\tau_0=2\alpha-N$, $\nu$ be a positive measure and  $u$ be a nonnegative
solution of (\ref{eq 1.1}) verifying
$$u(x)\ge c_j|x|^{\tau_{j}},\quad \forall x\in \R^N\setminus B_{4r_0}(0)$$
for some $c_j>0$ and $j\le j_0-2$.
Then for $p\in (0,\frac{N }{N -2\alpha})$, there exists $c_{j+1}>0$   such that
$$u(x)\ge c_{j+1} |x|^{\tau_{j+1}},\quad \forall x\in \R^N\setminus B_{4r_0}(0).$$

\end{lemma}
\begin{proof} For $r>\max\{1,4r_0\}$, let $O_r=   B_r(0)\setminus B_{4r_0}(0)$ and  $v_r$ be the unique solution of
\begin{equation}\label{eq 3.1}
  \arraycolsep=1pt
\begin{array}{lll}
 \displaystyle  (-\Delta)^\alpha   v_r(x) =  c_j^p|x|^{\tau_{j}p}\chi_{O_r}(x),\quad
 \forall  x\in \Omega,\\[2mm]
 \phantom{  (-\Delta)^\alpha   }
 v_r(x)=0,\qquad \forall  x\in \R^N\setminus\Omega,\\[2mm]
 \phantom{\,    }
\displaystyle   \lim_{|x|\to\infty} v_r(x)=0,
\end{array}
\end{equation}
where $\chi_{O_r}=1$ in $O_r$ and $\chi_{O_r}=0$ in $\R^N\setminus O_r$.
By Lemma \ref{cp}, we have that for any $r>1$,
$$u\ge v_r\quad {\rm in}\quad \R^N.$$
 From (\ref{3.2}), we have that
$$G_{\alpha,\Omega}(x,y)\ge \frac{ c_4 }{|x-y| ^{N-2\alpha}},\quad \forall x,\,y\in \R^N\setminus B_{4r_0}(0),\, x\not=y.$$

We observe that
\begin{eqnarray*}
v_r(x) &=&\mathbb{G}_{\alpha}[ c_j^p|\cdot|^{\tau_{j}p}\chi_{O_r}](x) \\
    &\ge &  c_4 c_j^p \int_{O_r} \frac{   |y|^{\tau_jp}}{|x-y| ^{N-2\alpha}}\, dy
    \\ &=& c_4 c_j^p |x|^{\tau_{j+1}}\int_{O_{\frac{r}{|x|}}(0)\setminus O_{\frac1{|x|}}(0)} \frac{ |z|^{\tau_jp}  }{|e_x-z|^{N-2\alpha}}\, dz
    \\&\to& c_4  c_j^p |x|^{\tau_{j+1}}\int_{\R^N\setminus B_{\frac{4r_0}{|x|}}(0)}\frac{ |z|^{\tau_jp}  }{|e_x-z|^{N-2\alpha}}\, dz\quad{\rm as}\quad r\to+\infty,
\end{eqnarray*}
where $e_x=\frac{x}{|x|}$, $\tau_jp<-2\alpha$ and for $x\in \R^N\setminus B_{4r_0}(0)$,
\begin{eqnarray*}
\int_{\R^N\setminus B_{\frac{4r_0}{|x|}}(0)}\frac{ |z|^{\tau_jp} }{|e_x-z|^{N-2\alpha}}\, dz &\ge & \int_{\R^N \setminus B_1(0)}\frac{ |z|^{\tau_jp}  }{|e_x-z| ^{N-2\alpha}}\, dz \\
   &=& \int_{\R^N \setminus B_1(0)}\frac{ |z|^{\tau_jp}  }{|e_N-z| ^{N-2\alpha}}\, dz.
\end{eqnarray*}
Let
$$c_{j+1}=c_4c_j^p  \int_{\R^N\setminus B_1(0)}\frac{ |z|^{\tau_jp} }{|e_N-z|^{N-2\alpha}} dz,$$
then
$$u(x)\ge c_{j+1}|x|^{\tau_{j+1}},\qquad\forall x\in \R^N\setminus B_{4r_0}(0).$$
The proof is complete.\end{proof}

Now we are ready to prove Theorem \ref{teo 1}$(i)$.

\medskip

\noindent{\bf Proof of Theorem \ref{teo 1} $(i)$.}
By contradiction, we assume that (\ref{eq 1.1}) has a very weak solution   $u\ge 0$.
Since $\nu\not=0$, there exists $n_1>4r_0$ such that
$$B_{n_1}(0)\cap {\rm supp}\nu\not=\emptyset,$$
then we have that
$$  \mathbb{G}_{\alpha, B_{n_1}}[\nu\chi_{B_{n_1}(0)}]>0\quad {\rm in}\quad O_{n_1},$$
and by Lemma \ref{cp}, we have that
$$u\ge \mathbb{G}_{\alpha,\Omega}[\nu ]\ge \mathbb{G}_{\alpha, B_{n_1}}[\nu\chi_{B_{n_1}(0)}]\quad {\rm in}\quad \R^N.$$
Let $\mu=\mathbb{G}_{\alpha, B_{n_1}}^p[\nu\chi_{B_{n_1}(0)}]>0$ in $B_{n_1}(0)$, then
$$u\ge \mathbb{G}_\alpha[\mu].$$

For $x\in \R^N\setminus B_{2n_1}(0)$,  we have that
\begin{eqnarray*}
\mathbb{G}_{\alpha,\Omega}[\mu](x)  \ge   c_5 \int_{B_{n_1}(0)\setminus B_{4r_0}(0)}\frac{\mu(y)dy}{|x-y|^{N-2\alpha}}
   \ge  c_6  \norm{\mu}_{L^1(B_{n_1}(0)\setminus B_{4r_0}(0) )} |x|^{2\alpha-N}.
\end{eqnarray*}
Thus, there exists $c_0>0$ such that
\begin{equation}\label{3.1-1}
 u(x)\ge c_0|x|^{\tau_0},\qquad \forall\,x\in \R^N\setminus B_{4r_0}(0)
\end{equation}
with $\tau_0=2\alpha-N<0$.
Then it implies by Lemma \ref{lm 2.2}  that for any $j\le j_0-1$,
\begin{equation}\label{2.4-1}
 u(x)\ge c_j|x|^{\tau_j},\qquad x\in \R^N\setminus B_{2r_0}(0),
\end{equation}
where $\{\tau_j\}_j$ is given   by (\ref{2.1}) and $c_j>0$.

Let $v_r$ the solution of (\ref{eq 3.1}) with $j=j_0-1$, we have that for any $r>8r_0$,
$$u(x)\ge v_r(x),\qquad \forall x\in\R^N.$$
Then for any $x\in B_{8r_0}(0)\setminus B_{4r_0}(0)$, $y\in B_{r}(0)\setminus B_{8r_0}(0) $, we have that $|x-y|\le 2|y|$  and
\begin{eqnarray*}
u(x) &\ge &c_3c_{j_0-1}^p \int_{B_r(0)\setminus B_{4r_0}(0)}\frac{|y|^{\tau_{j_0-1}p}  }{|x-y|^{N-2\alpha}}\,dy  \\
    &\ge & c_7  \int_{B_r(0)\setminus B_{4r_0}(0)} |y|^{2\alpha-N+\tau_{j_0-1}p}\,  dy
   \\ &\ge &\arraycolsep=1pt
\left\{\begin{array}{lll}
c_{8}[r^{\tau_{j_0}}-(4r_0)^{\tau_{j_0}}]\quad
  &{\rm if}\quad  \tau_{j_0}>0\\[2mm]
 \phantom{    }
\displaystyle   c_8[\log r -\log (4r_0)]\quad &{\rm if}\quad  \tau_{j_0}=0
\end{array}
\right.
\\&\to&\infty\qquad{\rm as}\quad r\to+\infty,
\end{eqnarray*}
which contradicts that $u\in L^1(\R^N,\frac{dx}{1+|x|^{N+2\alpha}})$ from the definition of very weak solution of (\ref{eq 1.1}).
  The proof ends.\hfill$\Box$

\medskip

\subsection{Half space}

We first recall the Green's estimate of the fractional Laplacian in half type space .

\begin{lemma}\label{lm 2.4}
Assume that  $N>\alpha$ and $\Omega\supset \R^N_+$,
then there exists $c_9>0$  such that
 \begin{equation}\label{eq 3.3}
 G_{\alpha,\Omega}(x,y)\ge \frac{c_9}{|x-y|^{N-2\alpha}}\min\left\{1,\left(\frac{x_Ny_N }{|x-y|^2}\right)^{\alpha}\right\},\quad\forall x,\,y\in \R^N_+.
 \end{equation}

\end{lemma}
\begin{proof}  From Corollary \ref{cr 3}, we have that
$$G_{\alpha,\Omega}(x,y)\ge G_{\alpha,\R^N_+}(x,y),\quad x,\,y \in \R^N,\ x\not=y.$$
From \cite[Corollary 1.4]{CT}, there exists $c_{10}>1$ such that for any $x,y\in \R^N_+$, $x\not=y$,
\begin{equation}\label{e 3.3}
 \frac1{c_{10}}\min\left\{1,\left(\frac{x_Ny_N }{|x-y|^2}\right)^{\alpha}\right\} \le  G_{\alpha,\R^N_+}(x,y)|x-y|^{N-2\alpha} \le c_{10} \min\left\{1,\left(\frac{x_Ny_N }{|x-y|^2}\right)^{\alpha}\right\},
 \end{equation}
which implies (\ref{eq 3.3}).
The proof is complete.
\end{proof}

Let $\mathcal{C}_r$ be the cone
$$\mathcal{C}_r=\bigcup_{t>r} \left\{(x',t)\in\R^{N-1}\times\R: \,   |x'|<t \right\}. $$

\begin{lemma}\label{lm 2.2}
Let   $\{\tau_j\}_j$ be given by (\ref{2.1}) with $\tau_0= \alpha-N$, $\nu$ be a positive measure and  $u$ be a nonnegative
solution of (\ref{eq 1.1}) satisfying
$$u(x)\ge c_j|x|^{\tau_{j}},\quad \forall x\in \mathcal{C}_1$$
for some $c_j>0$ and $j\le j_0-2$.
Then for $p\in (0,\frac{N+\alpha}{N-\alpha})$, there exists $c_{j+1}>0$   such that
$$u(x)\ge c_{j+1} |x|^{\tau_{j+1}},\quad \forall x\in \mathcal{C}_1.$$

\end{lemma}
\begin{proof} For $r>1$, let $O_r=\mathcal{C}_1\cap  B_r(0)$ and  $v_r$ be the unique solution of
\begin{equation}\label{eq 2.1}
  \arraycolsep=1pt
\begin{array}{lll}
 \displaystyle  (-\Delta)^\alpha   v_r(x) =  c_j^p|x|^{\tau_{j}p}\chi_{O_r}(x),\quad
 \forall\,  x\in \Omega,\\[2mm]
 \phantom{  (-\Delta)^\alpha   }
 v_r(x)=0,\qquad\ \forall\, x\in \R^N\setminus\Omega,\\[2mm]
 \phantom{\,    }
\displaystyle   \lim_{|x|\to+\infty} v_r(x)=0,
\end{array}
\end{equation}
where $\chi_{O_r}=1$ in $O_r$ and $\chi_{O_r}=0$ in $\R^N\setminus O_r$.
By Lemma \ref{cp}, we have that for any $r>1$,
$$u\ge v_r\quad {\rm in}\quad \R^N.$$
 From (\ref{eq 3.3}), we have that
\begin{eqnarray*}
v_r(x) &=&\mathbb{G}_{\alpha}[ c_j^p|\cdot|^{\tau_{j}p}\chi_{O_r}](x) \\
    &\ge &  c_9 c_j^p \int_{O_r} |y|^{\tau_jp}   \frac{1}{|x-y|^{N-2\alpha}}\min\left\{1,\left(\frac{x_Ny_N }{|x-y|^2}\right)^{\alpha}\right\} \,dy
    \\ &=& c_9 c_j^p |x|^{\tau_{j+1}}\int_{O_{\frac{r}{|x|}}(0)\setminus O_{\frac1{|x|}}(0)} \frac{ |z|^{\tau_jp} }{|e_x-z|^{N-2\alpha}}\min\left\{1,\left(\frac{ z_N }{|e_x-z|^2}\right)^{\alpha}\right\} \,dz
    \\&\to& c_9  c_j^p |x|^{\tau_{j+1}}\int_{\mathcal{C}_1\setminus B_{\frac1{|x|}}(0)}
    \frac{ |z|^{\tau_jp} }{|e_x-z|^{N-2\alpha}}\min\left\{1,\left(\frac{ z_N }{|e_x-z|^2}\right)^{\alpha}\right\}\, dz\quad{\rm as}\quad r\to+\infty,
\end{eqnarray*}
where $e_x=\frac{x}{|x|}$, $\tau_jp<-2\alpha$ and for any $x\in \mathcal{C}_1$,
\begin{eqnarray*}
 &&\int_{\mathcal{C}_1\setminus B_{\frac1{|x|}}(0)}\frac{ |z|^{\tau_jp} }{|e_x-z|^{N-2\alpha}}\min\left\{1,\left(\frac{ z_N }{|e_x-z|^2}\right)^{\alpha}\right\}\, dz
  \\&\ge& \int_{\mathcal{C}_1}\frac{ |z|^{\tau_jp} }{|e_x-z|^{N-2\alpha}}\min\left\{1,\left(\frac{ z_N }{|e_x-z|^2}\right)^{\alpha}\right\}\, dz
  \\&\ge&\int_{\mathcal{C}_1}\frac{ |z|^{\tau_jp} }{(1+|z|)^{N-2\alpha}} \frac{ z_N^\alpha }{(1+|z|)^{2\alpha} }\, dz.
\end{eqnarray*}

Let
$$c_{j+1}=c_9c_j^p  \int_{\mathcal{C}_1}\frac{ |z|^{\tau_jp} }{(1+|z|)^{N-2\alpha}} \frac{ z_N^\alpha }{(1+|z|)^{2\alpha} }\, dz,$$
then
$$u(x)\ge c_{j+1}|x|^{\tau_{j+1}},\qquad\forall x\in \mathcal{C}_1.$$
The proof is complete.\end{proof}

Now we are ready to prove Theorem \ref{teo 1}$(ii)$.

\medskip

\noindent{\bf Proof of Theorem \ref{teo 1} $(ii)$.}
By contradiction, we assume that (\ref{eq 1.1}) has a very weak  solution $u\ge 0$.

We first claim that there exists $c_0>0$ such that
\begin{equation}\label{3.1}
 u(x)\ge c_0|x|^{\alpha-N},\qquad\forall x\in \mathcal{C}_1.
\end{equation}
Indeed, let $\{O_n\}_n$ be a sequence of $C^2$ domain such that
$$\Omega\cap B_n(0)\subset O_n\subset \Omega\cap B_{n+1}(0).$$
 Let $n_2>1$ such that
$$B_{n_2}(0)\cap {\rm supp}\nu\not=\emptyset,$$
then we have that
$$  \mathbb{G}_{\alpha, O_{n_2}}[\nu\chi_{B_{n_2}(0)}]>0\quad {\rm in}\quad O_{n_2}$$
and
$$u\ge \mathbb{G}_{\alpha,\Omega}[\nu\chi_{B_{n_2}(0)} ]\ge \mathbb{G}_{\alpha, O_{n_2}}[\nu\chi_{B_{n_2}(0)}]\quad {\rm in}\quad \R^N.$$

Let $\mu=\mathbb{G}_{\alpha, O_{n_2}}^p[\nu\chi_{B_{n_2}(0)}]$ and then
$$u\ge \mathbb{G}_\alpha[\mu].$$

For $x\in \mathcal{C}_1$ and $y\in O_{n_1}\cap \{z\in\R^N:\, z_N>1\}$, we have that  $|x-y|\le |x|+|y|<(n_2+1)|x|$ and $x_N>\frac{\sqrt{2}}{2}|x|$
\begin{eqnarray*}
\min\left\{1,\left(\frac{ x_Ny_N }{|x-y|^2}\right)^{\alpha}\right\}   &\ge&   \min\left\{1,\left(\frac{ x_N  }{(n_2+1)^2|x|^2}\right)^{\alpha}\right\}
  \\& \ge &     \frac{c_{11}}{1+|x|^\alpha},
\end{eqnarray*}
  and
\begin{eqnarray*}
\mathbb{G}_{\alpha,\Omega}[\mu](x)  &\ge&   c_{11} \int_{O_{n_2} }\frac{\mu(y)dy}{|x-y|^{N-2\alpha}}\min\left\{1,\left(\frac{ x_Ny_N }{|x-y|^2}\right)^{\alpha}\right\} dy
  \\& \ge & c_{11}  \norm{\mu}_{L^1(O_{n_2} )} |x|^{\alpha-N},
\end{eqnarray*}
where $c_{11}>0$ depends on $n_2$.

From (\ref{3.1}),  we have that
$$u(x)\ge c_0|x|^{\tau_0},\qquad\forall\, x\in \mathcal{C}_1$$
with $\tau_0=\alpha-N<0$.
Then it implies by Lemma \ref{lm 2.2}  that for any $j\le j_0-1$,
\begin{equation}\label{2.4}
 u(x)\ge c_j|x|^{\tau_j},\qquad\forall\, x\in \mathcal{C}_1,
\end{equation}
where $\{\tau_j\}_j$ is given by (\ref{2.1}) and $c_j>0$.

Let $v_r$ the solution of (\ref{eq 2.1}) with $j=j_0-1$, we have that for any $r>1$,
$$u(x)\ge v_r(x),\qquad \forall x\in\R^N.$$
Then for any $x\in \mathcal{C}_1\setminus B_r(0)$ with $r>1$, $y\in \mathcal{C}_1\setminus B_{2|x|}(0)$, $|x-y|\le 2|y|$,   we have that
\begin{eqnarray*}
u(x) &\ge &c_9c_{j_0-1}^p \int_{O_r(0)\setminus B_1(0)}|y|^{\tau_{j_0-1}p}   \frac{1}{|x-y|^{N-2\alpha}}\min\left\{1,\left(\frac{x_Ny_N }{|x-y|^2}\right)^{\alpha}\right\} \, dy  \\
    &\ge & c_{11} x_N^\alpha \int_{O_r(0)\setminus B_{2|x|}(0)} |y|^{-N+\tau_{j_0-1}p}y_N^\alpha \, dy
   \\ &\ge &\arraycolsep=1pt
\left\{\begin{array}{lll}
c_{12}[r^{\tau_{j_0}}-(2|x|)^{\tau_{j_0}}]\quad
  &{\rm if}\quad  \tau_{j_0}>0\\[2mm]
 \phantom{    }
\displaystyle   c_{12}[\log r -\log (2|x|)]\quad &{\rm if}\quad  \tau_{j_0}=0
\end{array}
\right.
\\&\to&\infty\qquad{\rm as}\quad r\to+\infty,
\end{eqnarray*}
which contradicts that $u\in L^1(\R^N,\frac{dx}{1+|x|^{N+2\alpha}})$ from the definition of very weak solution of (\ref{eq 1.1}).
  The proof ends.\hfill$\Box$

\medskip

\subsection{Nonexistence in the classical setting}

In this subsection, we prove the nonexistence of classical solutions of semi-linear elliptic equations (\ref{eq 1.2}) and (\ref{eq 1.3}) by using the method
 in the proof of Theorem \ref{teo 1}. The main difference is that we use strong Comparison Principle replacing the weak one.

\smallskip

\noindent{ \bf Proof of Corollary \ref{cr 1}.}
Since $u\ge0$ is a classical solution of (\ref{eq 1.2}), then
if there exists one point $x_0\in\R^N\setminus B_{r_0}(0)$ such that $u(x_0)=0$, then
we have that
$$\int_{\R^N}\frac{u(y)}{|x_0-y|^{N+2\alpha}}dy=0,$$
which implies that
$$u\equiv0.$$
 So we may assume that $u>0$ in $\Omega$ and   let  $u_l(x)=u(l^{-1}x)$ for  $x\in\R^N$, $l>1$, then
$u_l$ is a positive solution of
$$
  \arraycolsep=1pt
\begin{array}{lll}
 \displaystyle   (-\Delta)^\alpha   u_l = l^{2\alpha} u_l^p\quad
 & {\rm in}\quad  \R^{N}\setminus B_{lr_0}(0),\\[2mm]
 \phantom{  (-\Delta)^\alpha }
\displaystyle    u_l \ge 0\quad
 & {\rm in}\quad   B_{lr_0}(0).
\end{array}
$$
We see that the positive function $w_l:=\mathbb{G}_{\alpha,\R^N\setminus \overline{B_{2lr_0}(0)}}[(l^{2\alpha}-1)u_l^p \chi_{B_{4lr_0}(0)\setminus B_{2lr_0}(0)}]$ is a classical solution of
  \begin{equation}\label{eq 5.1}
  \arraycolsep=1pt
\begin{array}{lll}
 \displaystyle   (-\Delta)^\alpha   w_l = l^{2\alpha} u_l^p\quad
 & {\rm in}\quad  \R^{N}\setminus \overline{B_{2lr_0}(0)},\\[2mm]
 \phantom{  (-\Delta)^\alpha }
\displaystyle    w_l =0 \quad
 & {\rm in}\quad   \overline{B_{2lr_0}(0)},
 \\[2mm]
 \phantom{   }
\displaystyle  \lim_{|x|\to+\infty} w_l(x)=0.
\end{array}
  \end{equation}
 The remaind of the  proof is similar to the proof of Theorem 1.1 $(i)$ just replacing the weak Comparison Principle by strong Comparison Principle, so we just sketch the proof.
 By strong Comparison Principle, we have that
$$u_l\ge w_l\quad{\rm in}\quad {\R^N}.$$
By Lemma \ref{lm 2.3} with $\Omega=\R^N\setminus B_{2lr_0}(0)$,
$$w_l(x)\ge c_{13}|x|^{2\alpha-N},\quad |x|>2lr_0,$$
thus,
$$u_l\ge c_{13}|x|^{2\alpha-N},\quad |x|>2lr_0.$$
By   Lemma \ref{cp} and  repeat the argument of the proof of Theorem 1.1 $(i)$ to obtain that
$$u_l(x)=+\infty\quad {\rm for}\quad |x|\ge 2lr_0,$$
which contradicts that $u$ is a classical solution of (\ref{eq 5.1}).\hfill$\Box$

\medskip

\noindent{ \bf Proof of Corollary \ref{cr 2}.}  Since $u\ge0$ is a classical solution of (\ref{eq 1.3}), then
if there exists one point $y_0\in\Omega$ such that $u(y_0)=0$, then
we have that
$$\int_{\R^N}\frac{u(y)}{|x_0-y|^{N+2\alpha}}\, dy=0,$$
which implies that
$$u\equiv0.$$
So we may assume that $u>0$ in $\Omega$. For $l>1$, let $u_l(x)=u(l^{-1}x)$ for $x\in\R^N$, then
$u_l$ is a positive solution of
 \begin{equation}\label{eq 5.2}
  \arraycolsep=1pt
\begin{array}{lll}
 \displaystyle   (-\Delta)^\alpha   u_l = l^{2\alpha} u_l^p\quad
 & {\rm in}\quad  \R^{N-1}\times(0,+\infty),\\[2mm]
 \phantom{  (-\Delta)^\alpha }
\displaystyle    u_l \ge0\quad
 & {\rm in}\quad  \R^{N-1}\times(-\infty,0].
\end{array}
 \end{equation}
The remaind  is similar to the proof of Theorem 1.1 $(ii)$  with $p<\frac{N+\alpha}{N-\alpha}$, and we omit here.  \hfill$\Box$

\section{Existence in the supercritical case  }

To prove Theorem \ref{teo 2}, the following estimate plays an important role  in the construction of the upper bound in the procedure of finding the solution.

\begin{lemma}\label{lm 6.1}
For $p\in \left[\frac{N+\alpha}{N-\alpha}, \frac{N}{N-2\alpha}\right)$, we have that
\begin{equation}\label{5.1}
\mathbb{G}_{\alpha,\R^N_+}[\mathbb{G}_{\alpha,\R^N_+}^p[\delta_{e_N}]]\le c_{14}\mathbb{G}_{\alpha,\R^N_+}[\delta_{e_N}]\quad{\rm in}\quad\R^N_+.
\end{equation}
\end{lemma}
\begin{proof}
From (\ref{e 3.3}), we have that for $x\in\R^N_+$, $x\not=e_N$,
\begin{equation}\label{e 5.3}
 \frac{1}{c_{10}|x-e_N|^{N-2\alpha}} \frac{x_N^\alpha}{1+|x|^{2\alpha}}\le  \mathbb{G}_{\alpha,\R^N_+}[\delta_{e_N}](x)  \le \frac{c_{10}}{|x-e_N|^{N-2\alpha}} \frac{x_N^\alpha}{1+|x|^{2\alpha}}.
 \end{equation}
Our aim here is to prove (\ref{5.1}) in $\R^N_+$, which is divided into  $D_1:= B_{\frac12}(e_N)$,
 $D_2= \{z\in\R^N_+:\, z_N<\frac14\}$,   $D_3:= \{z\in\R^N_+:\, z_N\ge\frac14,\, |z|>8\}$ and $D_4=\R^N_+\setminus (D_1\cap D_2\cap D_3)$.
Since  $D_4$ is compact and  $ \mathbb{G}_{\alpha,\R^N_+}[\mathbb{G}_{\alpha,\R^N_+}^p[\delta_{e_N}]]$ has no singularity and decaying, then
 (\ref{5.1}) holds in $D_4$.

{\it Case 1: $x\in D_1$. }
For $x\in D_1$, we observe that
\begin{eqnarray*}
  \mathbb{G}_{\alpha,\R^N_+}[\mathbb{G}_{\alpha,\R^N_+}^p[\delta_{e_N}]](x) &\le&  c_{15}\int_{\R^N_+}\frac{G_{\alpha,\R^N_+}(x,y)}{ |y-e_N|^{(N-2\alpha)p}(1+|y|)^{\alpha p}}\, dy,
\end{eqnarray*}
where
\begin{eqnarray*}
 && \int_{B_2(0)} \frac{G_{\alpha,\R^N_+}(x,y)}{ |y-e_N|^{(N-2\alpha)p}(1+|y|)^{\alpha p}}\, dy
 \\&\le& \int_{B_2(0)}\min\left\{1,\left(\frac{x_Ny_N }{|x-y|^2}\right)^{\alpha}\right\}\frac{c_{10}}{|x-y|^{N-2\alpha}}\frac{1}{|y-e_N|^{(N-2\alpha)p}}\,dy \\
 &\le & \int_{B_2(0)}\frac{c_{10}}{|(x-e_N)-y|^{N-2\alpha}}\frac{1}{ |y|^{(N-2\alpha)p}}\, dy \\
 &=& |x-e_N|^{2\alpha-(N-2\alpha)p} \int_{B_{\frac2{|x-e_N|}}(0)}\frac{c_{10}}{|e-z|^{N-2\alpha}}\frac{1}{ |z|^{(N-2\alpha)p}} \,dz
 \\&\le & |x-e_N|^{2\alpha-(N-2\alpha)p} \left[ \int_{B_1(e)} \frac{c_{10}}{|e-z|^{N-2\alpha}}dz+\int_{B_1(0)} \frac{c_{10}}{|z|^{(N-2\alpha)p}}dz \right. \\&& \left.+\int_{B_{\frac2{|x-e_N|}(0)}} \frac{c_{10}}{1+|z|^{(N-2\alpha)(p+1)}} dz\right]
 \\& \le& c_{16} |x-e_N|^{2\alpha-(N-2\alpha)p}\left(1+|x-e_N|^{(N-2\alpha)p-2\alpha} \right)
 \\& =& c_{16}|x-e_N|^{2\alpha-(N-2\alpha)p}+c_{15}
\end{eqnarray*}
and
\begin{eqnarray*}
 && \int_{\R^N\setminus  B_2(0)} \frac{G_{\alpha,\R^N_+}(x,y)}{ |y-e_N|^{(N-2\alpha)p}(1+|y|)^{\alpha p}}\, dy
 \\&\le& \int_{\R^N\setminus B_2(0)}\min\left\{1,\left(\frac{x_Ny_N }{|x-y|^2}\right)^{\alpha}\right\}\frac{c_{10}}{|x-y|^{N-2\alpha}}\frac{1}{|y|^{(N-\alpha)p}}\, dy \\
 &\le &    \int_{\R^N}\frac{c_{10}}{1+|y|^{(N-\alpha)(p+1)}}\, dy\le c_{17},
\end{eqnarray*}
here $e_N=\frac{x-e_N}{|x-e_N|}$ and $(N-\alpha)(p+1)>N$.
Therefore, (\ref{5.1}) holds for $x\in D_1$.

{\it Case 2: $x\in D_2$. } We note that
\begin{eqnarray*}
  \mathbb{G}_{\alpha,\R^N_+}[\mathbb{G}_{\alpha,\R^N_+}^p[\delta_{e_N}]](x) &\le&  c_{17}\int_{\R^N_+}\frac{G_{\alpha,\R^N_+}(x,y)y_N^\alpha}{ |y-e_N|^{(N-2\alpha)p}(1+|y|)^{2\alpha p}} dy.
\end{eqnarray*}
For $x\in D_2$ satisfying $|x|\ge \frac12$, let $D_{x}=\{z\in\R^N_+:\, z_N>2x_N\}$, then we have that
\begin{eqnarray*}
 && \int_{D_x} \frac{G_{\alpha,\R^N_+}(x,y)}{ |y-e_N|^{(N-2\alpha)p}(1+|y|)^{\alpha p}}\, dy
 \\&\le& x_N^\alpha \int_{D_{x} }\frac{y_N^\alpha  }{1+|x-y|^{2\alpha}} \frac{c_{10}}{|x-y|^{N-2\alpha}}\frac{1}{|y-e_N|^{(N-\alpha)p}}\,dy \\
 &\le & x_N^\alpha |x|^{(\alpha-N)(p+1)}\int_{\R^N_+}\frac{c_{10}}{|e_x-z|^{N-\alpha}}\frac{1}{|z-\frac{e_N}{|x|}|^{(N-\alpha)p}}\, dy \\
 &\le & c_{18} x_N^\alpha |x|^{(\alpha-N)(p+1)}
\end{eqnarray*}
and
\begin{eqnarray*}
  \int_{\R^N_+\setminus  D_{x}} \frac{G_{\alpha,\R^N_+}(x,y)}{ |y-e_N|^{(N-2\alpha)p}(1+|y|)^{\alpha p}}\, dy
 &\le& \int_{\R^N\setminus D_x(0)} \frac{x_N^\alpha y_N^\alpha}{1+|x-y|^{2\alpha}} \frac{c_{10}}{|x-y|^{N-2\alpha}} \frac{1}{1+|y|^{(N-\alpha)p}}\,dy \\
 &\le & c_{19} x_N^\alpha |x|^{(\alpha-N)(p+1)},
\end{eqnarray*}
where $(\alpha-N)(p+1)<-N$.

For $x\in D_2$ satisfying $|x|< \frac12$, we have that
\begin{eqnarray*}
 && \int_{D_x} \frac{G_{\alpha,\R^N_+}(x,y)}{ |y-e_N|^{(N-2\alpha)p}(1+|y|)^{\alpha p}}\, dy
 \\&\le& x_N^\alpha \int_{D_{x} }\frac{y_N^\alpha  }{1+|x-y|^{2\alpha}} \frac{c_{10}}{|x-y|^{N-2\alpha}}\frac{1}{|y-e_N|^{(N-\alpha)p}}\, dy \\
 &\le &c_{20} x_N^\alpha
\end{eqnarray*}
and
\begin{eqnarray*}
  \int_{\R^N_+\setminus  D_{x}} \frac{G_{\alpha,\R^N_+}(x,y)}{ |y-e_N|^{(N-2\alpha)p}(1+|y|)^{\alpha p}} dy
 &\le& \int_{\R^N\setminus D_x(0)} \frac{x_N^\alpha y_N^\alpha}{1+|x-y|^{2\alpha}} \frac{1}{|x-y|^{N-2\alpha}} \frac{1}{1+|y|^{(N-\alpha)p}}dy \\
 &\le & c_{21} x_N^\alpha.
\end{eqnarray*}
 Therefore, (\ref{5.1}) holds for $x\in D_2$.

{\it Case 3: $x\in D_3$. }   We see that
\begin{eqnarray*}
 && \int_{\R^N_+} \frac{G_{\alpha,\R^N_+}(x,y)}{ |y-e_N|^{(N-2\alpha)p}(1+|y|)^{\alpha p}} dy
 \\&\le& \int_{\R^N_+} \frac{c_{10}}{|x-y|^{N-2\alpha}}\frac{1}{ |y-e_N|^{(N-2\alpha)p}(1+|y|)^{\alpha p} } dy \\
 &= & c_{22}|x|^{2\alpha-(N-\alpha)p}  \int_{\R^N_+} \frac{c_{10}}{|e_x-z|^{N-2\alpha}}\frac{1}{ |z-\frac{e_N}{|x|}|^{(N-2\alpha)p}(|x|^{-1}+|z|)^{\alpha p} } dz\\
 &\le & c_{22}|x|^{2\alpha-(N-\alpha)p}  \left[ \int_{B_{\frac12}(e_N)} \frac{c_{10}}{|e_N-z|^{N-2\alpha}}dz+\int_{B_{\frac12}(\frac{e_N}{|x|})} |z-\frac{e_N}{|x|}|^{(N-2\alpha)p} dz \right. \\&& \left.+\int_{B_{\frac2{|x-e_N|}(0)}} \frac{c_{10}}{1+|z|^{(N-\alpha)(p+1)-\alpha}} dz\right]
\\ &\le & c_{23}|x|^{2\alpha-(N-\alpha)p},
\end{eqnarray*}
where $(N-\alpha)(p+1)-\alpha>N$.
Since $2\alpha-(N-\alpha)p\le \alpha-N$, then
(\ref{5.1}) holds for $x\in D_3$.
The proof ends. \end{proof}

\medskip

\noindent {\bf Proof of Theorem \ref{teo 2}.}   We first define the iterating sequence
$$v_0= k \mathbb{G}_{\alpha,\R^N_+}[\delta_{e_N}]>0$$
and
$$
 v_n  =  \mathbb{G}_{\alpha,\R^N_+}[v_{n-1}^p]+ k \mathbb{G}_{\alpha,\R^N_+}[\delta_{e_N}].
$$
Observing that
$$v_1= \mathbb{G}_{\alpha,\R^N_+}[(kv_0)^p] + k \mathbb{G}_{\alpha,\R^N_+}[\delta_{e_N}]>v_0$$
and  assuming that
$$
v_{n-1} \ge  v_{n-2} \quad{\rm in} \quad \R^N_+\setminus\{e_N\},
$$
we deduce that
\begin{eqnarray*}
 v_n =   \mathbb{G}_{\alpha,\R^N_+}[v_{n-1}^p]+ k \mathbb{G}_{\alpha,\R^N_+}[\delta_{e_N}]
 \ge  \mathbb{G}_{\alpha,\R^N_+}[v_{n-2}^p]+ k \mathbb{G}_{\alpha,\R^N_+}[\delta_{e_N}]
 =   v_{n-1}.
\end{eqnarray*}
Then the sequence $\{v_n\}_n$ is  increasing with respect to $n$.
Moreover, we have that
\begin{equation}\label{4.2.3}
\int_{\R^N_+} v_n(-\Delta)^\alpha  \xi \,dx =\int_{\R^N_+} v_{n-1}^p\xi \,dx +k\xi(e_N), \quad \forall \xi\in C^\infty_c(\R^N_+).
\end{equation}

We next build an upper bound for the sequence $\{v_n\}_n$.  For  $t>0$, denote
\begin{equation}\label{wt}
 w_t=t k^p\mathbb{G}_{\alpha,\R^N_+}[\mathbb{G}_{\alpha,\R^N_+}^p[\delta_{e_N}]]+ k\mathbb{G}_{\alpha,\R^N_+}[\delta_{e_N}]\le (c_{14}t k^p+ k)\mathbb{G}_{\alpha,\R^N_+}[\delta_{e_N}],
\end{equation}
where $c_{14}>0$ is from Lemma \ref{lm 6.1},
 then
\begin{eqnarray*}
  \mathbb{G}_{\alpha,\R^N_+}[w_t^p]+k\mathbb{G}_{\alpha,\R^N_+}[\delta_{e_N}]\le  (c_{14}t k^p+ k)^p\mathbb{G}_{\alpha,\R^N_+}[\mathbb{G}_{\alpha,\R^N_+}^p[\delta_{e_N}]] +  k  \mathbb{G}_{\alpha,\R^N_+}[\delta_{e_N}]
   \le  w_t,
\end{eqnarray*}
if
$$
 (c_{14}t k^p+ k)^p\le tk^p,
$$
that is,
\begin{equation}\label{4.2.4}
  (c_{14}t k^{p-1} + 1)^p\le t.
\end{equation}

Let $k_p=\left(\frac1{c_{14} p}\right)^{\frac1{p-1}}\frac{p-1}p$ and $t_p=\left(\frac p{p-1}\right)^{p}$, then if $k\le k_p$ and $t=t_p$, (\ref{4.2.4}) holds.
 Hence,  by the definition of $w_{t_p}$, we have  $w_{t_p}>v_0$ and
$$v_1= \mathbb{G}_{\alpha,\R^N_+}[v_0^p]+ k \mathbb{G}_{\alpha,\R^N_+}[\delta_0]<\mathbb{G}_{\alpha,\R^N_+}[w_{t_p}^p]+ k \mathbb{G}_{\alpha,\R^N_+}[\delta_0]=w_{t_p}.$$
Inductively, we obtain that
\begin{equation}\label{2.10a}
v_n\le w_{t_p}
\end{equation}
for all $n\in\N$. Therefore, the sequence $\{v_n\}_n$ converges. Let $u_{k}:=\lim_{n\to\infty} v_n$. By \eqref{4.2.3}, we have that $u_{k}$ is a very weak solution of (\ref{eq 1.2}).

We claim that $u_{k}$ is the minimal solution of (\ref{eq 1.1}), that is, for any nonnegative solution $u$ of (\ref{eq 1.2}), we always have $u_{k}\leq u$. Indeed,  there holds
\[
 u  = \mathbb{G}_{\alpha,\R^N_+}[  u^p]+ k \mathbb{G}_{\alpha,\R^N_+}[\delta_0]\ge v_0,
\]
 then
\[
 u  = \mathbb{G}_{\alpha,\R^N_+}[ u^p]+ k \mathbb{G}_{\alpha,\R^N_+}[\delta_0]\ge \mathbb{G}_{\alpha,\R^N_+}[ v_0^p]+ k \mathbb{G}_{\alpha,\R^N_+}[\delta_0]=v_{1}.
\]
We may show inductively that
\[
u\ge v_n
\]
for all $n\in\N$.  The claim follows.

 Similarly,  if problem (\ref{eq 1.2}) has a nonnegative solution $u$  for $ k_1>0$, then (\ref{eq 1.2}) admits a minimal solution $u_{k}$ for all $ k\in(0, k_1]$. As a result, the mapping $ k\mapsto u_{k}$ is increasing.
So we may define
$$k^*=\sup\{k>0:\ (\ref{eq 1.1})\ {\rm has\ minimal\ solution\ for\ }k \}$$
and we have that
$$k^*\ge k_p.$$

{\it Regularity of the very weak solution of (\ref{eq 1.2}). }  Let $u$ be a very weak solution of (\ref{eq 1.2}), take $\bar x=(\bar x_1,\cdots, \bar x_N) \in \R^N_+\setminus\{e_N\}$
and $r=\frac14\min\{|\bar x-e_N|, \bar x_N\}$, then
\begin{eqnarray*}
u  &=& \mathbb{G}_{\alpha,\R^N_+}[u^p]+k\mathbb{G}_{\alpha,\R^N_+}[\delta_{e_N}]  \\
   &=&  \mathbb{G}_{\alpha,\R^N_+}[u^p\chi_{B_r(\bar x )}]+\mathbb{G}_{\alpha,\R^N_+}[u^p\chi_{\R^N_+\setminus B_r(\bar x )}] +k\mathbb{G}_{\alpha,\R^N_+}[\delta_{e_N}],
\end{eqnarray*}
where $\mathbb{G}_{\alpha,\R^N_+}[\delta_0]$ is $C^\infty_{loc}(\R^N_+\setminus\{e_N\})$.
To be convenient,  we write $B_i=B_{2^{-i}r}(\bar x )$.
For $x\in B_{i}$, we have that
$$\mathbb{G}_{\alpha,\R^N_+}[\chi_{\R^N_+\setminus B_{i-1}} u^p](x)=\int_{\R^N_+\setminus B_{i-1}}   u(y)^pG_{\alpha,\R^N_+}(x,y)dy,$$
then, for some $C_i>0$, we have that
\begin{equation}\label{3.01-1}
\norm{\mathbb{G}_{\alpha,\R^N_+}[\chi_{\R^N_+\setminus B_{i}}  u^p]}_{C^2(B_{i-1})}\le C_i \norm{ u^p}_{L^1(B_{2r}(\bar x ))}
\end{equation}
and for some   constant $\tilde{c}_i>0$ depending on $i$, we obtain that
 \begin{equation}\label{3.01}
\norm{\mathbb{G}_{\alpha,\R^N_+}[\delta_0]}_{C^2(B_{i-1})} \le   \tilde{c}_i|\bar x |^{2\alpha-N}.
\end{equation}
By  Proposition 2.2 in \cite{CQ}, we have that $u^p\in L^{q_0}(B_{2r_0}(\bar x ))$ with $q_0=\frac12(1+\frac1p\frac{N}{N-2\alpha})>1$
and then
$$\mathbb{G}_{\alpha,\R^N_+}[\chi_{B_{2r}(\bar x )} u^p]\in L^{p_1}(B_{2r}(\bar x ))\quad {\rm with }\ \ p_1=\frac{Nq_0}{N-2\alpha q_0}.$$
Similarly,
$$ u^p\in L^{q_1}(B_{r}(\bar x ))\quad {\rm with }\ \   q_1=\frac{p_1}{p}$$
and
$$\mathbb{G}_{\alpha,\R^N_+}[\chi_{B_{r}(\bar x )}u^p]\in L^{p_2}(B_{r}(\bar x ))\quad {\rm with }\ \ p_2=\frac{Nq_1}{N-2\alpha q_1}.$$
Let $q_i=\frac{p_{i}}{p}$ and $p_{i+1}=\frac{Nq_i}{N-2q_i}$ if $N-2q_i>0$. Then we obtain inductively that
$$ u^p\in L^{q_i}(B_{i}) \quad\mbox{and}
\quad \mathbb{G}_{\alpha,\R^N_+}[\chi_{B_{i}} u^p]\in L^{p_{i+1}}(B_{i}).$$
We may verify that
$$\frac{q_{i+1}}{q_i}=\frac1p\frac N{N-2\alpha q_i}>\frac1p\frac{N}{N-2\alpha q_1}>1.$$
Therefore,
$\lim_{i\to+\infty} q_i=+\infty,$
so there exists $i_0$ such that $N-2q_{i_0}>0$, but $N-2q_{i_0+1}<0$, then we deduce that
$$\mathbb{G}_{\alpha,\R^N_+}[\chi_{B_{i_0}} u^p]\in L^{\infty}(B_{{i_0}}).$$
As a result, we obtain that
$$ u\in L^{\infty}(B_{i_0}).$$
By  regularity results in \cite{RS1}, we know from (\ref{3.01}) that $u$ is H\"{o}lder continuous in $B_{{i_0}}$ and  so is $ u^p$. Then $u$ is a classical solution of  (\ref{eq 1.1}).
  \hfill$\Box$

\smallskip

\bigskip

\noindent{\bf Acknowledgements:}  H. Chen  is supported by NNSF of China, No:11401270, 11661045 and by  Jiangxi Provincial Natural Science Foundation,
 No: 20161ACB20007.   The author would like to thank for the support by CIM in Nankai University and  useful discussion with  Professors Yiming Long and
  Feng Zhou.

\end{document}